\documentclass{amsart}
\usepackage{latexsym,amssymb,amsthm,amsmath}

\theoremstyle{plain}
\newtheorem{theorem}{Theorem}[section]
\newtheorem*{Theorem B}{Theorem B}
\newtheorem*{Theorem A}{Theorem A}

\newtheorem{lemma}{Lemma}[section]

\newtheorem{proposition}{Proposition}[section]

\newtheorem{corollary}{Corollary}[section]

\newtheorem{example}{Example}[section]
\numberwithin{equation}{section}

\theoremstyle{remark}
\newtheorem{remark}{Remark}[section]

 \numberwithin{equation}{section}

\def\<{\left < }
\def\>{\right >}
\def\({\left ( }
\def\){\right )}
\def\o{\omega }

\def\e{\eqref}
\def\p{\partial }

\def\x{{\bf x}}
\def\k{\kappa}

\begin{document}

\markboth{B.-Y. Chen}{Ricci solitons and concurrent vector fields}

\title[Ricci solitons and concurrent vector fields]{Ricci solitons and concurrent vector fields}

\author[ B.-Y. Chen and S. Deshmukh]{Bang-Yen Chen and Sharief Deshmukh }

 \address{Department of Mathematics\\Michigan State University \\619 Red Cedar Road \\East Lansing, MI 48824--1027, USA}

\email{bychen@math.msu.edu}

 \address{Department of Mathematics\\  King Saud University\\  Riyadh 11451, Saudi Arabia}
 \email{shariefd@ksu.edu.sa}

\begin{abstract} A Ricci soliton $(M^n,g,v,\lambda)$ on a Riemannian manifold $(M^n,g)$ is said to have concurrent potential field if its potential field $v$ is a concurrent vector field.
 In the first part of this paper we completely classify Ricci solitons with concurrent potential fields. In the second part  we derive a necessary and sufficient condition for a submanifold to be a Ricci soliton in a Riemannian manifold equipped with a concurrent vector field. In the last part, we classify shrinking Ricci solitons with $\lambda=1$  on Euclidean hypersurfaces.
 Several applications of our results are also presented.
\end{abstract}



\maketitle

\section{Introduction}

A smooth vector field $\xi $ on a Riemannian manifold $(M,g)$ is said to define a {\it Ricci soliton} if it satisfies
\begin{equation}\label{1.1}
\frac{1}{2}{\mathcal L}_{\xi }g+Ric=\lambda g,
\end{equation}
where ${\mathcal L}_{\xi }g$ is the Lie-derivative of the metric tensor $g$ with respect to $\xi $, $Ric$ is the Ricci tensor of $(M,g)$ and $\lambda $ is a constant. We shall denote a Ricci soliton by $(M,g,\xi ,\lambda )$. We call the vector field $\xi $ the {\it potential field} of the Ricci soliton. The Ricci
soliton $(M,g,\xi ,\lambda )$ is called  {\it shrinking, steady} or {\it expanding} according to  $\lambda >0,\, \lambda =0,$ or $\lambda<0$, respectively. A Ricci soliton $(M,g,\xi ,\lambda )$ is said to be {\it trivial} if $(M,g)$ is an Einstein manifold.

A Ricci soliton $(M,g,\xi ,\lambda )$ is called a {\it gradient Ricci soliton} if its potential field $\xi $ is the gradient of some smooth function $f$ on $M$. We shall denote a gradient Ricci soliton by $(M,g,f,\lambda )$ and call
the smooth function $f$ the {\it potential function}. A gradient Ricci soliton $(M,g,f,\lambda )$ is called {\it trivial} if its potential function $f$ is a constant. It follows from \e{1.1} that trivial gradient Ricci solitons are trivial Ricci solitons automatically since $\xi=\nabla f$. 
It is well-known that if $(M,g,\xi ,\lambda )$ is a compact Ricci soliton, then the potential field $\xi $ is a gradient of some smooth function $f$ up to the addition of a Killing field and thus a compact Ricci soliton is a gradient Ricci soliton (cf. \cite{P}). 

During the last two decades, the geometry of Ricci solitons has been the focus of attention of many
mathematicians. In particular, it has become more important after Grigory Perelman  applied Ricci solitons to solve the long standing Poincar\'e conjecture posed in 1904. G. Perelman observed in \cite{P} that  the Ricci solitons on compact simply connected Riemannian manifolds   are gradient Ricci solitons as solutions of Ricci flow. 

If the holonomy group of a Riemannian $m$-manifold $M$ leaves a point invariant, then it was proved in \cite{Yano} that there exists a vector field $v$ on $M$ which  satisfies 
\begin{align}\label{1.2} \nabla_Z v=Z\end{align} 
for any  vector $Z$ tangent to $M$, where $ \nabla$ denotes the Levi-Civita connection of $M$.
Such a vector field is called a {\it concurrent vector field}. 
Riemannian manifolds equipped with concurrent vector fields have been studied by many mathematician (see, e.g. \cite{CY,MM,MR,PRV1,PRV2,Yano,YC}). Concurrent vector fields have also been studied in Finsler geometry since the beginning of 1950s (see, e.g. \cite{ME,T}). 

There are two aspects of the study of Ricci solitons, one looking at the influence on the topology by
the Ricci soliton structure of the Riemannian manifold (see e.g. \cite{Der,LG})
and the other looking at its influence on its geometry (see e.g. \cite{CD,Des,DAA}).
 In this paper we are interested in the geometry of Ricci solitons arisen from concurrent vector fields on Riemannian manifolds. 
 
 In the first part of this paper we completely classify Ricci solitons with concurrent potential fields. In the second part  we derive a necessary and sufficient condition for a submanifold to be a Ricci soliton in a Riemannian manifold equipped with a concurrent vector field. In the last part, we classify shrinking Ricci solitons with $\lambda=1$  on Euclidean hypersurfaces.
 Several applications of our results are also presented.

\section{Preliminaries}

\subsection{Basic formulas and definitions for submanifolds}

For general references on Riemannian submanifolds, we refer to \cite{book73,book,book14}.

 Let  $(N^m,\tilde g)$ denote an $m$-dimensional Riemannian manifold and $\phi:M^n \to N^m$ an isometric immersion from an $n$-dimensional Riemannian manifold $(M^n,g)$ into $(N^m,\tilde g)$.      
Denote by $\nabla$ and $\tilde\nabla$ the Levi-Civita connections on $(M^n,g)$ and $(N^m,\tilde g)$, respectively. 

For vector fields $X,Y$ tangent to $M^n$ and $\eta$ normal to $M^n$, the formula of Gauss and the formula of Weingarten are given respectively by \begin{align} &\label{2.1}\tilde \nabla_XY=\nabla_XY+h(X,Y), \;\;
\\& \label{2.2}\tilde \nabla_X \eta=-A_\eta X+D_X\eta,\end{align} 
where $\nabla_X Y$ and $h(X,Y)$ are the tangential and the normal components of $\tilde\nabla_X Y$. Similarly,  $-A_\eta X$  and  $D_X\eta$ are the tangential and normal components of  $\tilde \nabla_X \eta$. These two formulas define the second
fundamental form $h$, the shape operator $A$, and the normal connection $D$ of $M^n$ in the ambient space $N^m$. 
 
  For a normal vector $\eta\in T_p^{\perp}M$ at $p\in M$,  $A_{\eta}$ is a self-adjoint endomorphism of the tangent space $T_pM$. The shape operator and the second fundamental form are related by
 \begin{align} &\label{2.3}\tilde g(h(X,Y),\eta)=g(A_{\eta}X,Y).\end{align}
  The {\it mean curvature vector}  $H$ of $M^n$ in $N^m$ is defined by \begin{align}\label{2.4} H=\(\frac{1}{n}\){\rm trace}\, h.\end{align}
The equations of Gauss and Codazzi are given respectively by
\begin{align} &\label{2.5} g(R(X,Y)Z,W) = \tilde g( \tilde R(X,Y)Z,W)  + \tilde g(h(X,W),h(Y,Z))\\&\notag \hskip1.6in - \tilde g(h(X,Z),h(Y,W)),\\
& \label{2.6} (\tilde R(X,Y)Z)^\perp =(\bar \nabla_X h)(Y,Z)-(\bar\nabla_Y h)(X,Z),\end{align}
for vectors $X,Y,Z,W$  tangent to $M$ and $\zeta,\eta$  normal to $M$, where $(\tilde R(X,Y)Z)^\perp $ is the normal component of $\tilde R(X,Y)Z$ and $\bar\nabla h$ is defined by
\begin{equation}\begin{aligned}&\label{2.8} (\bar\nabla_X h)(Y,Z) = D_X h(Y,Z) - h(\nabla_X Y,Z) - h(Y,\nabla_X Z).\end{aligned}\end{equation} 

\subsection{Examples of Riemannian manifolds endowed with concurrent fields}
The best known example of Riemannian manifolds endowed with concurrent vector fields is the Euclidean space with the concurrent vector field given by its position vector field ${\bf x}$ (with respect to the origin).

For more general examples of Riemannian manifolds with concurrent vector fields, let us consider warped product manifolds of the form:  $I\times_{s} F$, where $I$  is an open interval of the real line $\bf R$ with $s$ as its arclength and $F$ is a Riemannian manifold. The metric tensor $g$ of $I\times_{s} F$ is given by
$g=ds^2+s^2 g_F$,
where $g_F$ is the metric tensor of the second factor $F$. Let us put
$v= s\frac{\p}{\p s}$.
It follows easily from Proposition 4.1 of \cite[page 79]{book} that the vector field $v$ satisfies  \e{1.2} for any vector $Z$ tangent  to $I\times_{s} F$. Therefore $I\times_{s} F$ admits a concurrent vector field $v=s{\p}/{\p s}$.

\section{Ricci solitons with concurrent potential fields}

A Ricci soliton $(M^n,g,v,\lambda)$ on a Riemannian manifold $(M^n,g)$ is said to have {\it concurrent potential field} if its potential field $v$ is a concurrent vector field.

The following theorem classifies Ricci solitons on Riemannian manifolds endowed with a concurrent potential field.

\begin{theorem} \label{T:3.1} A Ricci soliton $(M^n,g,v,\lambda)$ on a Riemannian $n$-manifold $(M^n,g)$ has concurrent  potential field $v$ if and only if the following two conditions hold:

\begin{enumerate}
\item[{\rm (a)}] The Ricci soliton is a shrinking Ricci soliton with $\lambda=1$.

\item[{\rm (b)}]  $M^n$ is an open part of a warped product manifold $I\times_s F$, where $I$ is an open interval  with arclength $s$  and $F$ is an Einstein $(n-1)$-manifold whose Ricci tensor satisfies $Ric_F=(n-2)g_F$, $g_F$ is the metric tensor of $F$. 
\end{enumerate}
\end{theorem}
\begin{proof} Assume that $(M^n,g,v,\lambda)$ is a Ricci soliton on a Riemannian $n$-manifold equipped with a concurrent  potential field $v$. Then we have
\begin{align} \label{3.1}& \nabla_X v=X,\;\; \forall X\in TM^n.\end{align}
It follows from \e{3.1} that the concurrent vector field $v$ vanishes on a measure zero subset of $M^n$  at most. By applying \e{3.1} and the definition of sectional curvature, it is easy to verify that the sectional curvature of $M^n$ satisfies
\begin{align} \label{3.2}& K(X, v)=0. \end{align}
for each unit vector $X$ orthogonal to $v$. Hence the Ricci tensor of $M^n$ satisfies 
\begin{align} \label{3.3}& Ric(v, v)=0. \end{align}

Let us put $v=\mu e_1$, where $e_1$ is a unit vector field tangent to $M^n$. Also let us extend $e_1$ to a local orthonormal frame $\{ e_1,\ldots, e_n\}$ on $M^n$. Denote by $\{\omega^1,\ldots,\omega^n\}$ the dual frame of 1-forms of $\{ e_1,\ldots, e_n\}$.

Define the connection forms $\omega_i^j\, (i,j=1,\ldots,n)$ on $M^n$ by
\begin{align} \label{3.4}& \nabla_X e_i=\sum_{j=1}^n \omega_i^j(X)e_j,\;\; i=1,\ldots,n.\end{align}
From \e{3.1} with $X=e_1$, \e{3.4} and the continuity we find
\begin{align} \label{3.5}& e_1\mu =1,
\\&\label{3.6} \nabla_{e_1}e_1=0.\end{align}

Put $\mathcal D_1={\rm Span}\{e_1\}$ and $ \mathcal D_2={\rm Span}\{e_2,\ldots,e_n\}.$
It follows from \e{3.6} that $\mathcal D_1$ is a totally geodesic distribution so that the leaves of $\mathcal D_1$ are geodesics of $M^n$.
Also, we may derive from \e{3.1} with $X=e_i\, (i=2,\ldots,n)$ that
\begin{align} \label{3.7}& e_2\mu=\cdots=e_n \mu=0,
\\&\label{3.8} \mu\omega^1_i(e_i)=-1,\;\;
\\& \label{3.9} \omega^1_j(e_i)=0,\;\; i\ne j.\end{align}

From Cartan's structure equations, we have
\begin{align} \label{3.10}& d\omega^i=-\sum_{j=1}^n \omega^i_j\wedge \omega^j,\;\; i=1,\ldots,n. \end{align}
Thus, after applying \e{3.9} and \e{3.10}, we obtain $d\omega^1=0$. Hence we have locally $\omega^1=ds$ for some function $s$ on $M^n$. 
It follows from \e{3.9} that
\begin{align} \label{3.11}& g([e_i,e_j],e_1)=\omega^1_j(e_i)-\omega^1_i(e_j)=0,\;\; 2\leq i\ne j\leq n.\end{align}
Therefore $\mathcal D_2$ is an integrable distribution. 
Moreover,  from \e{3.8} we know that the second fundamental form $\hat h$ of each leaf $L$ of $\mathcal D_2$ in $M^n$ satisfies \begin{align} \label{3.12}\hat h(e_i,e_j)=-\frac{\delta_{ij}}{\mu}e_1,\;\;  2\leq i,j\leq n,\end{align} 
which shows that the mean curvature of each leaf $L$ is given by $-\mu^{-1}$. 

Equation \e{3.12} implies that  each leaf of $\mathcal D_2$ is a totally umbilical hypersurface of $M^n$ whose mean curvature vector is $\hat H=-e_1/\mu$.  Furthermore, by applying \e{3.7} we conclude that $\mathcal D_2$ is a spherical distribution, i.e., the mean curvature vector of each totally umbilical leaf is parallel in the normal bundle. Consequently,  a result of S. Hiepko (see, e.g., \cite[page 90]{book}) implies that  $M^n$ is locally a warped product manifold $I \times_{f(s)}  F$ whose warped metric is given by
\begin{align}\label{3.13} g=ds^2+f^2(s)g_{F}\end{align}
such that $e_1=\p/\p s$. 

It follows from \e{3.13} that the sectional curvature of $M^n$ satisfies 
\begin{align}\label{3.14} K(X,v)=-\frac{f''(s)}{f(s)}\end{align}
for each unit vector $X$ orthogonal to $v$. Now, after comparing \e{3.2} with  \e{3.14} we obtain $f''(s)=0$. Therefore we obtain $f(s)=as+b$ for some constants $a$ and $b$.

If $a=0$ holds,  then the warped product manifold $I \times_{f(s)}  F$ is a Riemannian product, which implies that every leaf of $\mathcal D_2$ is totally geodesic in $M^n$. Hence $\mu$ must be zero, which contradicts to  \e{3.12}. Therefore we must have $a\ne 0$. Hence, after applying a suitable translation and dilation in $s$ we get $f(s)=s$. Consequently, $M^n$ is locally a warped product manifold $I\times_s F$. 

On the other hand, it follows from the definition of Lie-derivative and condition \e{3.1} that the Lie-derivative satisfies
\begin{align} \label{3.15}& ({\mathcal L}_v g)(X,Y)=g(\nabla_X v,Y)+g(\nabla_Y v,X)=2g(X,Y)\end{align}
for any $X,Y$ tangent to $M^n$. Combining \e{3.15} with \e{1.1} gives 
\begin{align} \label{3.16}& Ric(X,Y)=(\lambda-1)g(X,Y),\end{align}
which shows that $M^n$ is an Einstein $(n-1)$-manifold. After comparing \e{3.3} and \e{3.16} we conclude that $M^n$ is a Ricci flat space. Hence we get $\lambda=1$. Consequently, the Ricci soliton $(M^n,g,v,\lambda)$ is a shrinking one.

Since $M^n$ is a Ricci flat space, it follows from Corollary 4.1(3) of \cite[page 82]{book} or formula (9.109) of \cite[page 267]{Besse} that the second factor $F$ of the warped product manifold $I \times_f(s)  F$ is an Einstein manifold  satisfying $Ric_F=(n-2)g_F$.

The converse can be verified by direct computation.
\end{proof}

Theorem \ref{T:3.1} implies immediately the following

\begin{corollary} There do not exist steady or expanding Ricci solitons with concurrent potential fields.
\end{corollary}

\section{Riemannian submanifolds as Ricci solitons}

From now on, we make the following

\vskip.1in
\noindent{\bf Assumption.} {\it $(N^m,\tilde g)$ is a Riemannian $m$-manifold endowed with a concurrent vector field $v$.} 
\vskip.1in

For an isometric immersion $\phi:M^n\to N^m$ of a Riemannian $n$-manifold $(M^n,g)$ into $(N^m,\tilde g)$, we denote by $v^T$ and $v^\perp$ the tangential and normal components of $v$ on $M^n$, respectively. 
As before, we denote by $h, A$ and $D$ the second fundamental form, the shape operator and the normal connection of the submanifold $M^m$ in $N^m$, respectively. 

\begin{theorem}\label{T:4.1} A submanifold $M^n$ in $N^m$ admits a Ricci soliton $(M^n,g,v^T,\lambda)$ if and only if the Ricci tensor of $(M^n,g)$ satisfies
\begin{align} \label{4.1}Ric(X,Y)=(\lambda-1)g(X,Y)-\tilde g(h(X,Y),v^\perp)\end{align}
for any $X,Y$ tangent to $M^n$.
\end{theorem} 
\begin{proof} Let $\phi:M^n\to N^m$ denote the isometric immersion. We have 
\begin{align} \label{4.2} v=v^T+v^\perp.\end{align}
Since $v$ is a concurrent vector field on the ambient space $N^m$, it follows from \e{1.2}, \e{4.2} and formulas of Gauss and Weingarten that
\begin{equation}\begin{aligned} \label{4.3} X&=\tilde \nabla_X v^T+\tilde \nabla_X v^\perp
\\&= \nabla_X v^T+h(X,v^T)-A_{v^\perp}X+D_Xv^\perp
\end{aligned}\end{equation} for any $X$ tangent to $M^n$.

By comparing the tangential and normal components from \e{4.3} we obtain
\begin{align} \label{4.4}& \nabla_X v^T=A_{v^\perp}X+X,
\\&\label{4.5} h(X,v^T)=-D_X v^\perp.
\end{align}
From the definition of Lie derivative and \e{4.4} we obtain
\begin{equation}\begin{aligned} \label{4.6}({\mathcal L}_{v^T}g)(X,Y)&=g(\nabla_X v^T,Y)+g(\nabla_Y v^T,X) \\&=2g(X,Y)+2 g(A_{v^\perp}X,Y)
\\&=2g(X,Y)+2\tilde g(h(X,Y),v^\perp)
\end{aligned}\end{equation} for $X,Y$ tangent to $M^n$. Consequently,  by applying \e{1.1} and \e{4.5}, we conclude that $(M^n,g,v^T,\lambda)$ is a Ricci soliton if and only if we have 
\begin{equation}\begin{aligned} \label{4.7} Ric(X,Y)+g(X,Y)+ \tilde g(h(X,Y),v^\perp)=\lambda g(X,Y),
\end{aligned}\end{equation}
which is nothing but \e{4.1}.
\end{proof}

Recall that the position vector field $\x$ of a Euclidean $m$-space $\mathbb E^m$ is a concurrent vector field. The simplest examples of Ricci solitons $(M^n,g,v^T,\lambda)$ on submanifolds in a Riemannian manifold with concurrent field are the following ones.
 
\begin{example} \label{E:4.1} {\rm  Let $\gamma(s)$ be a unit speed curve lying in the unit hypersphere $S^{m-n}_o(1)$ of  $\mathbb E^{m-n+1}$ centered at the origin $o$. 
Consider the Riemannian submanifold $(M^n,g)$ of $\mathbb E^m$ defined by
$$\phi(s,x_2,\ldots,x_n)=(\gamma(s)x_2,x_2,x_3,\ldots,x_n).$$
Then $M^n$ is a flat space and $(M^n,g,\x^T,\lambda)$ is a shrinking Ricci soliton satisfying \e{4.1} with $\lambda=1$. Moreover, $\x^T=\x$ and $M^n$ is generated by lines in $\mathbb E^m$ through the origin $o$.
}\end{example}

The following provides more examples of Ricci solitons on submanifolds.

\begin{example} \label{E:4.2} {\rm Let $k$ be a natural number such that $2\leq k\leq n-1$ and $r=\sqrt{k-1}$.
Consider the spherical hypercylinder $\phi:S^k(r)\times \mathbb E^{n-k}\to  \mathbb E^{n+1}$ defined by
$\{({\bf y},x_{k+2},\ldots,x_{n+1})\in \mathbb E^{n+1}:  {\bf y}\in \mathbb E^{k+1}\;and \; \<{\bf y},{\bf y}\>=r^2\}.$
 It is straightforward to verify that the spherical hypercylinder $S^k(\!\sqrt{k-1})\times \mathbb E^{n-k}$  in $\mathbb E^{n+1}$ satisfies \e{4.1} with $\lambda=1$. Hence $(S^k(\!\sqrt{k-1})\times \mathbb E^{n-k},g,\x^T,\lambda)$ is a shrinking Ricci soliton with $\lambda=1$.

}\end{example}

\begin{example} \label{E:4.3} {\rm Let $n_1,\ldots,n_p$ be integers $\geq 2$ and $r_1,\ldots,r_p$ be positive numbers satisfying
$({n_1-1})/{r_1^2}=\cdots=({n_p-1})/{r_p^2}.$ Put $n=n_1+\cdots+n_p$.  

Let
$(M^n,g)$ denote the Riemannian product  $S^{n_1}(r_1)\times\ldots\times S^{n_p}(r_p)$ of $p$ spheres $S^{n_1}(r_1),\ldots,S^{n_p}(r_p)$ of radii $r_1,\ldots,r_p$, respectively, which is isometrically imbedded in $\mathbb E^{n+p}$ in the standard way. 
It is direct to verify that  $(M^n,g,\x^T,\lambda)$ is a shrinking Ricci soliton with $\lambda$ equal to $({n_1-1})/{r_1^2}$.

}\end{example}

\section{Some applications of Theorem \ref{T:4.1}}

A Riemannian submanifold $M^n$ is called {\it $\eta$-umbilical} (with respect to a normal vector field $\eta$) if its shape operated satisfies $A_\eta=\varphi I$, where  $\varphi$ is a function on $M^n$ and $I$ is the identity map.

The following two results are immediate consequences of Theorem \ref{T:4.1}.

\begin{theorem} \label{T:5.1} A Ricci soliton $(M^n,g,v^T,\lambda)$ on a submanifold $M^n$ in $N^m$ is trivial  if and only if $M^n$ is $v^\perp$-umbilical.
\end{theorem}

\begin{corollary}\label{C:5.1} Every Ricci soliton $(M^n,g,v^T,\lambda)$ on a totally umbilical submanifold $M^n$ of $N^m$ is a trivial Ricci soliton.
\end{corollary}

Following \cite{book}, the scalar curvature $\tau$ of a Riemannian $n$-manifold $(M^n,g)$ is defined to be\begin{align}\label{15.1} \tau=\sum_{1\leq i<j\leq n} K(e_i,e_j),\end{align}
where $\{e_1,\ldots,e_n\}$ is an orthonormal frame of $M^n$.

Another easy application  of Theorem \ref{T:4.1} is the following.

\begin{proposition} \label{P:5.1} If $(M^n,g,v^T,\lambda)$ is a Ricci soliton on a minimal submanifold $M^n$ in $N^m$, then  $M^n$ has constant scalar curvature given by $n(\lambda-1)/2$. 
\end{proposition}
\begin{proof} Assume that $(M^n,g,v^T,\lambda)$ is a Ricci soliton on a submanifold $M^n$ in $N^m$. Then 
Theorem \ref{T:4.1} implies that the Ricci tensor of $M^n$ satisfies
\begin{align} \label{15.2}Ric(X,Y)=(\lambda-1)g(X,Y)-\tilde g(h(X,Y),v^\perp)\end{align}
for $X,Y$ tangent to $M^n$. 
If $M^n$ is minimal in $N^m$, then the mean curvature vector vanishes identically. In particular, this implies that $\tilde g(H,v^\perp)=0$. Hence, we obtain \e{15.2}  that
$$\sum_{i=1}^n Ric(e_i,e_i)=n(\lambda-1).$$ Therefore, by \e{5.1}, $M^n$ has constant scalar curvature  $n(\lambda-1)/2$.
\end{proof}

Let $\nabla f$ denote the gradient of a function $f$ on $M^n$. By applying \e{4.4} and \e{4.5} we have the following.

\begin{lemma} \label{L:5.1} Let $M^n$ be a submanifold of $N^m$. Then we have
\begin{align}\label{5.1} &\nabla \psi=-A_{v^\perp}v^T,\\& \label{5.2} v^T=\nabla \varphi,
\end{align}
where $\psi=\frac{1}{2}\tilde g(v^\perp,v^\perp)$
and  $\varphi=\frac{1}{2}\tilde g(v,v)$.

\end{lemma}
\begin{proof} Let $M^n$ be a submanifold of $N^m$. Then we find from \e{4.5} that
\begin{equation}\begin{aligned} \notag X\psi&=\tilde g(\tilde \nabla_X v^\perp,v^\perp)=\tilde g(D_X v^\perp,v^\perp)=-g(A_{v^\perp}v^T,X),
\end{aligned}\end{equation} which implies \e{5.1}.
Equation \e{5.2} follows from 
\begin{equation}\begin{aligned} \notag X\varphi &=\tilde g(\tilde \nabla_X v,v)=\tilde g(X,v)=g(X,v^T)\end{aligned}\end{equation} for $X$ tangent to $M^n$.
\end{proof}

The next result follows immediately from \e{5.2} of Lemma \ref{L:5.1}.

\begin{proposition}\label{P:5.3} Every Ricci soliton $(M^n,g,v^T,\lambda)$ on a submanifold $M^n$ of $N^m$ is a gradient Ricci soliton with potential function $\varphi=\frac{1}{2}\tilde g(v,v)$.
\end{proposition}

This proposition shows that  the gradient Ricci soliton $(M^n,g,\varphi,\lambda)$ on $M^n$ is trivial if and only if $\tilde g(v,v)$ is constant on $M^n$.

\begin{corollary}\label{C:5.2} A gradient Ricci soliton $(M^n,g,\varphi,\lambda)$ on a submanifold $M^n$ of $N^m$ is trivial if and only if the concurrent vector field $v$ on $N^m$ is  normal to $M^n$.
\end{corollary}
\begin{proof} Let $M^n$ be a submanifold of $N^m$. Suppose that $(M^n,g,\varphi,\lambda)$  is a trivial gradient Ricci soliton. Then $\tilde g(v,v)$ is constant on $M^n$. Thus by taking the derivative of $\tilde g(v,v)$ with respect to a tangent vector $X$, we find $0=X\tilde g(v,v)=2g(X,v)$ according to \e{1.2}. Because this is true for any arbitrary tangent vector of $M^n$, the concurrent vector field $v$ must be normal to $M^n$. 

Conversely, if $v$ is normal to $M^n$, then we have $X\tilde g(v,v)=2g(X,v)=0$. Thus $\tilde g(v,v)$ is constant on $M^n$. Consequently,  the gradient Ricci soliton is a trivial one according to Corollary \ref{C:5.2}.
\end{proof}

The last result of this section is the following.

\begin{proposition}\label{P:5.3} If $(M^n,g,\x^T,\lambda)$ is a Ricci soliton on a hypersurface of $M^n$ of $\mathbb E^{n+1}$, then $M^n$ has at most two distinct principal curvatures given by
\begin{align}\label{5.5}\k_1,\k_2=\frac{n\alpha+\rho \pm \sqrt{(n\alpha+\rho)^2+4-4\lambda}}{2},\end{align}
where $\alpha$ is the mean curvature and $\rho$ is the support function, i.e., $H=\alpha N$ and $\rho=\<N,\x\>$ with $N$ being a unit normal vector field.\end{proposition}
\begin{proof}
Assume that $(M^n,g,\x^T,\lambda)$ is a Ricci soliton on a hypersurface of $M^n$ of $\mathbb E^{n+1}$, where $\x^T$ denotes the tangential component of the position vector field $\x$.  Let $\{e_1,\ldots,e_n\}$ be an orthonormal frame on $M^n$ such that $e_1,\ldots,e_n$ are eigenvectors of the shape operator $A_N$. Then we have
\begin{align}\label{5.6} A_Ne_i=\kappa_i e_i,\;\;\; i=1,\ldots,n.\end{align}
From equation \e{2.5} of Gauss we obtain
\begin{align}\label{5.7} Ric(X,Y)=n g_0(h(X,Y),H)-\sum_{i=1}^n g_0(h(X,e_i),h(Y,e_i)),\end{align}
where $g_0$ denotes the Euclidean metric of $\mathbb E^{n+1}$. 
It follows from \e{5.6}, \e{5.7} and Theorem \ref{T:4.1} that $(M^n,g,\x^T,\lambda)$ is a Ricci soliton if and only if we have
\begin{align}\label{5.8} (n\alpha -\k_j)\k_i\delta_{ij}=(\lambda-1)\delta_{ij}-\rho \k_i\delta_{ij},\end{align}
where $\delta_{ij}$ is the Kronecker delta. Equation \e{5.8} is equivalent to
\begin{align}\label{6.4} \k_i^2-(n\alpha +\rho)\k_i +\lambda-1=0,\;\;\; i=1,\ldots,0,\end{align}
which implies the proposition\end{proof}

\section{Shrinking Ricci solitons on Euclidean hypersurfaces}

The purpose of this section is to prove the following classification theorem.

\begin{theorem} \label{T:6.1} Let $(M^n,g,\x^T,\lambda)$ be a shrinking Ricci soliton on a hypersurface of $M^n$ of $\mathbb E^{n+1}$ with $\lambda=1$. Then $M^n$ is an open portion of one of the following hypersurfaces of $\mathbb E^{n+1}$:
\begin{enumerate}
\item[{\rm (1)}] A totally umbilical hypersurface;
\item[{\rm (2)}] A flat hypersurface generated by lines through the origin $o$ of $\mathbb E^{n+1}$; 
\item[{\rm (3)}] A spherical hypercylinder $S^k(\sqrt{k-1})\times \mathbb E^{n-k}$, $2\leq k\leq n-1$.\end{enumerate}\end{theorem}
\begin{proof} Assume that $(M^n,g,\x^T,\lambda)$ is a shrinking Ricci soliton on a hypersurface of $M^n$ of $\mathbb E^{n+1}$. Then it follows from Proposition \ref{P:5.3} that $M^n$ has at most two distinct principal curvatures given by
\begin{align} \label{6.1} \frac{n\alpha+\rho + \sqrt{(n\alpha+\rho)^2+4-4\lambda}}{2},\;\; \frac{n\alpha+\rho - \sqrt{(n\alpha+\rho)^2+4-4\lambda}}{2}.\end{align}
If $M^n$ has only one principal curvature, then $M^n$ is  totally umbilical. 

Now, let us assume that $M^n$ has two distinct principal curvatures and $\lambda=1$. Then \e{6.1} implies that the two distinct principal curvatures are given respectively by 0 and $n\alpha +\rho$. Let $\kappa$ denote the nonzero principal curvature, i.e., $\kappa=n\alpha+\rho$.
Let us assume that the multiplicities of  $\k$ and 0 are $k$ and $n-k$, respectively, for some $k$ with $1\leq k<n$. Then we have $n\alpha=k\k$. Hence the mean curvature $\alpha$ and the support function $\rho$ are related by
\begin{align}\label{6.2} n (1-k)\alpha=k\rho.\end{align}

\noindent {\it Case} (a): $k=1$. In this case, \e{6.2} gives $\rho=\tilde g(\x,N)=0$. Thus the concurrent vector field $\x$ is tangent to $M^n$. So, it follows from \e{1.2} that $\tilde \nabla_X\x=X$. Hence integral curves of $\x$ are part of lines through the origin in $\mathbb E^{n+1}$. Therefore we obtain case (2) of the theorem.
 
\noindent {\it Case} (b): $2\leq k\leq n-1$. Without loss of generality, we may assume that 
\begin{align}\label{6.3} A_N=\begin{pmatrix} \k I_k& 0\\ 0& 0_{n-k}\end{pmatrix}\end{align}
with respect to an orthonormal tangent frame $\{e_1,\ldots,e_n\}$ of $M^n$, where $I_k$ is an $k\times k$ identity matrix and $0_{n-k}$ is an $(n-k)\times (n-k)$ zero matrix. 
We put 
\begin{align}\label{6.4} \mathcal D_1={\rm Span}\{e_1,\ldots,e_k\},\;\; \mathcal D_2={\rm Span}\{e_{k+1},\ldots,e_n\}.\end{align}
By taking the derivative of \e{6.2} with respect a tangent vector $X$ of $M^n$, we find 
\begin{align}\label{6.5} X\alpha=-\frac{k}{n (1-k)}g(\x^T,A_NX)=\frac{k}{n (k-1)}g(A_N\x^T,X) .\end{align}
Thus we have
\begin{align}\label{6.6} \nabla\alpha=\frac{k}{n (k-1)}A_N\x^T ,\end{align}
which implies that the gradient $\nabla \alpha$ lies in the distribution $\mathcal D_1$. Therefore, without loss of generality, we may assume that 
\begin{align}\label{6.7} \nabla\alpha=\zeta e_1 \end{align}
for some function $\zeta$. So we have
\begin{align}\label{6.8} e_2\alpha =\cdots=e_n\alpha=e_2\k =\cdots=e_n\k=0. \end{align}

For any vector fields $X,Y\in \mathcal D_1$ and $V,W\in \mathcal D_2$, we have 
\begin{align}\label{6.9} h(X,Y)=\k g(X,Y),\;\; h(X,V)=h(V,W)=0. \end{align}
It follows from \e{2.8}, \e{6.9} and equation $(\bar \nabla_V h)(W,X)=(\bar\nabla_X h)(V,W)$ of Codazzi that $h(\nabla_V W,X)=0.$ Since this is true for any vector field $X$ in $\mathcal D_1$, we conclude from \e{6.3} that $\nabla_V W$ lies in $\mathcal D_2$. Therefore $\mathcal D_2$ is a totally geodesic integrable distribution, i.e., $\mathcal D_2$ is an integrable distribution whose leaves  are totally geodesic submanifolds of $M^n$. Moreover, it follows from $h(V,W)=0$ that each leaf of $\mathcal D_2$ is in fact a totally geodesic submanifold of $\mathbb E^{n+1}$. Consequently, $M^n$ are foliated by $(n-k)$-dimensional  totally geodesic submanifolds of $\mathbb E^{n+1}$.

For $1\leq i\ne j\leq k$ and $t\in \{k+1,\ldots,n\}$, we find from \e{2.8}, \e{6.3}, \e{6.8} and \e{6.9} that \begin{align}\label{6.10} &(\bar \nabla_{e_i} h)(e_j,e_t)=-h(e_j,\nabla_{e_i} e_t),
\;\; (\bar \nabla_{e_t} h)(e_i,e_j)=0. \end{align}
Thus from $(\bar \nabla_{e_i} h)(e_j,e_t)=(\bar\nabla_{e_t} h)(e_t,e_j)$, we obtain $\o^t_i(e_j)=0.$
Therefore $\mathcal D_1$ is also a totally geodesic integrable distribution. Consequently, the de Rham decomposition theorem implies that $M^n$ is locally a Riemannian product, say $M_1^k\times \mathbb E^{n-k}$, of a Riemannian $r$-manifold $M_1^k$ and the Euclidean $(n-k)$-space. Furthermore, due to $h(\mathcal D_1,\mathcal D_2)=\{0\}$ by \e{6.3}, Moore's lemma implies that the immersion is a direct product immersion, i.e.,
$$S^{k}\times \mathbb E^{n-k}\subset \mathbb E^{k+1}\times \mathbb E^{n-k},$$
where $S^{k}\subset \mathbb E^{k+1}$ is the standard imbedding of   a $k$-sphere. Consequently, we obtain case (3) of the theorem.
\end{proof}

\begin{remark} Further classification theorems for Ricci solutions on hypersurfaces will be given in another paper. Also, Ricci solitons with concircular potential fields will be discussed in a separate article.
\end{remark}

\noindent {\bf Acknowledgements.} This work is supported by NPST program of King Saud University Project Number 13-MAT1813-02.

\end{document}